\documentclass{article}

\usepackage[T1]{fontenc}
\usepackage{amsfonts}
\usepackage{makeidx}
\usepackage{amsmath}
\usepackage{amssymb}
\usepackage{amsthm}
\usepackage{mathrsfs}
\usepackage[final]{graphicx}
\usepackage[all]{xy}
\usepackage{float}
\usepackage{fancybox}
\usepackage{color}
\usepackage{authblk}

\usepackage{multicol}

\theoremstyle{plain}
\newtheorem{Thm}{Theorem}[section]         

\newtheorem{theorem}{Theorem}

\newtheorem{Lemma}[Thm]{Lemma}
\newtheorem{Cor}[Thm]{Corollary}
\newtheorem{Prop}[Thm]{Proposition}

\theoremstyle{definition}

\newtheorem{Rem}[Thm]{Remark}  

\theoremstyle{remark}

\newtheorem{Ex}[Thm]{Example}

\numberwithin{equation}{section}

\begin{document}

\title{Nash multiplicities and isolated points of maximum multiplicity}
\author{B. Pascual-Escudero
\footnote{The author was supported by MTM2012-35849.
}} 

\AtEndDocument{\bigskip{\footnotesize
  \textsc{Depto. Matem\'aticas,
Facultad de Ciencias, Universidad Aut\'onoma de Madrid 
and Instituto de Ciencias Matem\'aticas CSIC-UAM-UC3M-UCM, Cantoblanco 28049 Madrid, Spain} \par 
  \textit{E-mail address}, B. Pascual-Escudero: \texttt{beatriz.pascual.escudero@gmail.com} \par
}}

\date{\today}

\maketitle

\begin{abstract}
Let $X$ be an algebraic variety defined over a field of characteristic zero, and let $\xi \in \mathrm{\underline{Max}\; mult}(X)$ be a point in the closed subset of maximum multiplicity of $X$. We provide a criterion, given in terms of arcs, to determine whether $\xi $ is isolated in $\mathrm{\underline{Max}\; mult}(X)$. More precisely, we use invariants of arcs derived from the Nash multiplicity sequence to characterize when $\xi $ is an isolated point in $\mathrm{\underline{Max}\; mult}(X)$.
\end{abstract}


\normalsize

\section*{Introduction}

In many situations, arc spaces are useful when one is looking for some information about the singularities of varieties. Let $X$ be an algebraic variety over a field $k$, and $\xi \in X$ a point. An arc $\varphi $ in $X=\mathrm{Spec}(B)$ through $\xi $ is a morphism
$$\varphi ^*:\mathrm{Spec}(K[[t]])\longrightarrow X$$
mapping the closed point to $\xi $ or, equivalently, a homomorphism of rings $\varphi :B\longrightarrow K[[t]]$ such that $\varphi (\mathcal{P}_{\xi })\subset (t)$, where $\mathcal{P}_{\xi }\subset B$ is the prime ideal defining $\xi $. The order of the arc $\varphi $, which we will denote by $\mathrm{ord}(\varphi )$, is the greatest positive integer $n$ such that $\varphi (\mathcal{P}_{\xi })\subset (t^n)$. Here $K$ may be any field extension of $k$. Some examples of the study of the connections between arcs and singularities can be found, for instance, in the works of Ein, Ishii, Musta\c{t}\u{a}, Reguera and Yasuda among others.

\subsection*{The Nash multiplicity sequence}

We will be interested in a sequence of positive integers attached to an arc at a given point, the so called Nash multiplicity sequence, which was defined by M. Lejeune-Jalabert in \cite{L-J} for hypersurfaces, and generalized later by M. Hickel in \cite{Hickel05}. This sequence can be constructed as follows. Let $\xi $ be a point in $X$ and let $\varphi $ be an arc through $\xi $. Also let $\Gamma _0=\varphi \times i:B\times K[t]\rightarrow K[[t]]$ be the graph of $\varphi $, which is additionally an arc in $X_0=X\times \mathbb{A}^1$ through $\xi _0=(\xi ,0)$. Consider the following sequence of blow ups $\pi _i$ at points:
\begin{equation}\label{eq:Nms}
\xymatrix@R=15pt@C=50pt{
\mathrm{Spec}(K[[t]]) \ar[dd]^{\Gamma _0^*} \ar[ddr]^{\Gamma _1^*} \ar[ddrrr]^{\Gamma _l^*} & & & & \\
& & & & \\
X_0=X\times \mathbb{A}^1 & X_1 \ar[l]_>>>>>{\pi _1} & \ldots \ar[l]_{\pi _2} & X_l \ar[l]_{\pi _l} & \ldots \ar[l]_{\pi _{l+1}} \\
\xi _0=(\xi ,0) & \xi _1 & \ldots & \xi _l 
}\end{equation}
Here $\pi _i $ is the blowup at $\xi _{i-1}$, where $\xi _{i}=\mathrm{Im}(\Gamma _{i}^*)\cap \pi _{i}^{-1}(\xi _{i-1})$ for $i=1,\ldots ,l,\ldots $, and $\Gamma _i$ is the arc in $X_i$ through $\xi _i$ obtained by lifting $\Gamma _0$. Thus, the sequence is constructed by blowing up at closed points selected by the arcs $\Gamma _i$ for $i\geq 0$. Let us recall that the multiplicity of $X$ at a point $\eta \in X$ is given by an upper semicontinuous function
\begin{align*}
\mathrm{mult} (X): X & \longrightarrow \mathbb{N} \\
\eta & \longmapsto \mathrm{mult} (X)(\eta )=\mathrm{mult}_{\eta }(X)=\mathrm{mult}(\mathcal{O}_{X,\eta })\mbox{,}
\end{align*}
where $\mathrm{mult}(\mathcal{O}_{X,\eta })$ stands for the multiplicity of the local ring $\mathcal{O}_{X,\eta }$ at the maximal ideal. Let us recall that this multiplicity can be computed as $a_d\cdot d!$, where $a_d$ is the coefficient of the highest order term in the Hilbert polynomial of $\mathcal{O}_{X,\eta }$ (see \cite[Ch. 11]{Hu_Sw}). From a geometrical point of view, if $X$ is defined over $\mathbb{C}$, then the multiplicty of $X$ at $\eta $ corresponds to the smallest of the ranks over the generic fiber when considering all possible local morphisms $(X,\eta )\rightarrow (\mathbb{C}^d,0)$. If we denote by $m_i$ the multiplicity of $X_i$ at $\xi _i$ for $i=0,\ldots ,l,\ldots $, then the Nash multiplicity sequence of $\varphi $ in $X$ at $\xi $ is the sequence of positive integers
$$m_0\geq m_1\geq \ldots \geq m_l=m_{l+1}=...\geq 1$$
(see \cite[Section 2.2]{Br_E_P-E} for the detailed construction).

\vspace{0.2cm}

When $X$ is a hypersurface, the Nash multiplicity sequence of an arc $\varphi \in \mathcal{L}(X)$ can be regarded as a refinement of the multiplicity of $X$ at $\xi :=\varphi (\langle t\rangle )$ in the following sense: On the one hand, the multiplicity function defines a stratification of $X$ into locally closed subsets, and the multiplicity of $X$ at a point corresponds to that of the stratum containing it. If $X$ is a hypersurface embedded in a smooth scheme $V$, this stratification is determined by the order of vanishing of the partial derivatives applied to a local equation $f$ defining $X$. On the other hand, consider the spaces of $i$-jets of $X$, which we shall denote by $\mathcal{L}_i(X)$ for $i\geq 0$, and the natural truncations from the arc space: $\pi _{X,i}:\mathcal{L}(X) \longrightarrow \mathcal{L}_i(X)$. Hence, $\pi _{X,i}(\mathcal{L}(X))\subset \mathcal{L}_i(X)$, for $i\geq 0$, is the subset of $i$-jets in $X$ which are the truncation of some arc in $X$. In \cite{L-J}, M. Lejeune-Jalabert proved that for each $i\geq 0$ there is a stratification of $\pi _{X,i}(\mathcal{L}(X))\subset \mathcal{L}_i(X)$ into disjoint locally closed subsets: $\pi _{X,i}(\mathcal{L}(X))=\cup _{1\leq \mu _i\leq \ldots \leq \mu _0}\mathcal{H}_{\mu _0,\ldots ,\mu _i}$. The Nash multiplicity sequence $(m_0,\ldots ,m_i,\ldots )$ attached to the arc $\varphi $ is determined by these stratifications: for $j\geq 0$, $\pi _{X,j}(\varphi )\in \mathcal{H}_{m_0,\ldots ,m_j}$. In particular, for $j=0$, $\pi _{X,0}(\mathcal{L}(X))=\cup _{1\leq \mu _0}\mathcal{H}_{\mu _0}\subset X$, corresponds to the stratification of $X$ given by the multiplicity, so $m_0$ is just the multiplicity of $X$ at $\xi $.

\vspace{0.2cm}

Throughout this paper, we use Hickel's approach since we state and prove our results for general varieties, not just hypersurfaces.

\vspace{0.2cm}

\subsection*{Our results}

We will be particularly interested in the study of Nash multiplicity sequences of arcs through points of maximum multiplicity of $X$. Let $m:=\mathrm{max\, mult}(X)$ be the maximum value achieved by $\mathrm{mult}(X)$.  We write
$$\mathrm{\underline{Max}\, mult}(X)=\left\{ \eta \in X:\mathrm{mult}_{\eta }(X)\geq m\right\} =\left\{ \eta \in X:\mathrm{mult}_{\eta }(X)=m\right\} $$
for the closed subset of the singular locus of $X$ consisting of the points of highest multiplicity. If $X$ is a reduced equidimensional scheme, then $X$ is regular if and only if the multiplicity equals one at every point (see \cite[Section 2.10]{Br_V2}). This is why the closed subset $\mathrm{\underline{Max}\, mult}(X)$ is an object of interest in resolution of singularities.

\vspace{0.2cm}

If $X$ is defined over a field $k$ of characteristic zero, one can define the order of contact of an arc $\varphi $ (through $\xi $) with $\mathrm{\underline{Max}\, mult}(X)$, 
$$r_{X,\varphi }\in \mathbb{Q}_{\geq 1}\mbox{.}$$
This order of contact is an invariant of the arc $\varphi $ at the point $\xi $ in $X$. This invariant can be computed as the order of a particular Rees algebra (see \cite[Section 3]{Br_E_P-E}). From $r_{X,\varphi }$, one can obtain, for instance, the number of blow ups as in (\ref{eq:Nms}) that are needed before the Nash multiplicity sequence decreases for the first time: 
$$\rho _{X,\varphi }:=\mathrm{min}_{i\in \mathbb{Z}_{>0}}\left\{ i:m_i<m_0 \right\} =\left[ r_{X,\varphi } \right]\mathrm{.}$$
We call $\rho _{X,\varphi }$ the persistance of $\varphi $ in $X$. In fact, sometimes it can be even more interesting to consider the quotient $\bar{r}_{X,\varphi }=\frac{r_{X,\varphi }}{\mathrm{ord}(\varphi )}$, which we refer to as the normalized order of contact of $\varphi $ with $\mathrm{\underline{Max}\, mult}(X)$, since in this way one gets rid of the influence of the order of the arc. For the same reason, it is convenient to define also the normalized persistance of $\varphi $ in $X$ as $\bar{\rho }_{X,\varphi }=\frac{\rho _{X,\varphi }}{\mathrm{ord}(\varphi )}$. Indeed, the set 
$$\Phi _{X,\xi }=\left\{ \bar{r}_{X,\varphi } \right\} _{\varphi }\subset \mathbb{Q}_{\geq 1}\mbox{,}$$
where $\varphi $ runs over all arcs in $X$ through $\xi $, is an invariant of $X$ at $\xi $. This invariant reflects information about $\xi $ coming from arcs. In particular, it uses information provided by the Nash multiplicity sequences of those arcs.  It was proven in \cite[Theorem 4.2.5]{Br_E_P-E} that this set has a minimum, which turns out to be the invariant $\mathrm{ord}_{\xi }^{(d)}(X)$, Hironaka's order in dimension $d$, which plays a key role in constructive resolution (see for example \cite{Hir1}, \cite{E_V97} and \cite[Sections 13 and 25]{Br_V2}).

\vspace{0.2cm}

In this work we try to understand better what the set $\Phi _{X,\xi }$ can tell us about the singularities of X. We prove that the supremum of $\Phi _{X,\xi }$ actually allows us to determine whether $\xi $ is an isolated point of $\mathrm{\underline{Max}\, mult}(X)$ or not:

\begin{theorem}\label{thm:main}(Main Theorem) 
Let $X$ be a variety over a field $k$ of characteristic zero, and let $\xi $ be a point in $\mathrm{\underline{Max}\, mult}(X)$. Then $\xi $ is an isolated point of $\mathrm{\underline{Max}\, mult}(X)$ if and only if the set $\Phi _{X,\xi }$ is upper bounded.
\end{theorem}

In terms of the Nash multiplicity sequence it will mean (Corollary \ref{cor:persistance}) that, whenever $\xi $ is an isolated point of $\mathrm{\underline{Max}\, mult}(X)$, then no arc through $\xi $ can be found so that its normalized persistance in $X$ is higher than a given integer (depending on $X$ and $\xi $). On the other hand, if $\xi $ belongs to a component of $\mathrm{\underline{Max}\, mult}(X)$ of dimension $1$ or more, then there is no bound for how big $\bar{\rho }_{X,\varphi }$ will be for some arcs.

\vspace{0.2cm}

In the last section of this paper, we present an additional condition over $X$ and isolated points of $\mathrm{\underline{Max}\, mult}(X)$ under which the supremum of $\Phi _{X,\xi }$ can be computed (Proposition \ref{thm:upper_bound_bigtau}). As we will see, this condition is also related to an invariant of constructive resolution of singularities: the $\tau $ invariant (see \cite{B}). We will also show some illustrative examples there.

\vspace{0.2cm}

\textbf{Acknowledgments: }The author is very grateful to A. Bravo and S. Encinas for their guide and suggestions, to S. Ishii and L. Narv\'aez for fruitful conversations, and to J. M. Conde-Alonso for his suggestions for the shaping of this paper.

\section{The order of contact of $\varphi $ with $\mathrm{\underline{Max}\, mult}(X)$}\label{sec:notation}

In what follows, we will assume $X$ to be an algebraic variety of dimension $d$ over a field $k$ of characteristic zero such that $\mathrm{max\, mult}(X)=b$, and $\xi $ to be a point in $\mathrm{\underline{Max}\, mult}(X)$, which for simplicity we will assume to be closed. The notation used in this section will be the standard one through the rest of the paper. Details about Rees algebras in resolution and basic results used here can be found, for instance, in \cite{E_V}.

\vspace{0.2cm}

Let $R$ be a regular ring which is of finite type over $k$. For us, a Rees algebra over $R$ (or over $V=\mathrm{Spec}(R)$) is a graded ring $\mathcal{G}=\oplus _{i\in \mathbb{Z}_{\geq 0}}I_iW^i\subset R[W]$ with $I_0=R$, which is finitely generated as an $R$-algebra. Note that this definition is more general than the (usual) one considering only algebras of the form $R[IW]$ for some ideal $I\subset R$. The singular locus of $\mathcal{G}$, $\mathrm{Sing}(\mathcal{G})$, is the subset of $V$ composed by the points $\eta $ of $\mathrm{Spec}(R)$ for which $\nu _{\eta }(f)\geq i$ for all $fW^i\in \mathcal{G}$, where $\nu _{\eta }(f)$ denotes the order of $f$ in the regular local ring $R_{\eta }$. It can be shown that  $\mathrm{Sing}(\mathcal{G})$ is a closed subset of $V$. By the order of an element $fW^i\in\mathcal{G}$ at a point $\eta \in \mathrm{Sing}(\mathcal{G})$, we mean the quotient $\frac{\nu _{\eta }(f)}{i}=:\mathrm{ord}_{\eta }(fW^i)$. The order of $\mathcal{G}$ at $\eta \in \mathrm{Sing}(\mathcal{G})$ is defined as $\mathrm{ord}_{\eta }(\mathcal{G})=\mathrm{inf}_{fW^i\in \mathcal{G}}\left\{ \mathrm{ord}_{\eta }(fW^i)\right\} $, and can actually be computed as $\mathrm{ord}_{\eta }(\mathcal{G})=\mathrm{min}\left\{ \mathrm{ord}_{\eta }(fW^i)\right\} $, where the minimum runs over a finite set of generators of $\mathcal{G}$.

\vspace{0.2cm}

We say that a Rees algebra $\mathcal{G}$ over $R$ represents the multiplicity of $X\hookrightarrow V=\mathrm{Spec}(R)$ at $\xi \in \mathrm{\underline{Max}\, mult}(X)$ if
$$\mathrm{Sing}(\mathcal{G})=\mathrm{\underline{Max}\, mult}(X)$$
locally in a neighborhood of $\xi $, and this condition is stable under sequences of permissible transformations for $\mathcal{G}$, that is, after any sequence of the form
\begin{equation}\label{eq:seq_perm}
\xymatrix@R=0.3pc@C=5pt{
V & = & V_0 & & & V_1 \ar[lll]_{\pi _1} & & & \ldots \ar[lll]_{\pi _2} & & & V_l \ar[lll]_{\pi _l}\\
\mathcal{G} & = & \mathcal{G}_0 & & & \mathcal{G}_1 \ar[lll] & & & \ldots \ar[lll] & & & \mathcal{G}_l \ar[lll]
}
\end{equation}
where each $\pi _j$ is either a smooth morphism or a blow up with center a regular closed subset of $\mathrm{Sing}(\mathcal{G}_{j-1})$ for $j=1,\ldots ,l$, as long as $\mathrm{max\, mult}(X_l)=\mathrm{max\, mult}(X)$ if $X_l$ is the strict transform of $X_{l-1}$ in $V_l$ whenever $\pi _l$ is a blow up, and the pullback if it is a smooth morphism. Here, the transform $\mathcal{G}_j$ of $\mathcal{G}_{j-1}$, for $j=1,\ldots ,l$, is  
\begin{equation}\label{eq:trans_law_RA}
\mathcal{G}_j=\oplus _{i\geq 0}I_{i,j}W^i\mbox{\; where \; }I_{i,j}=I_{i,j-1}\mathcal{O}_{V_j}\cdot I(E_j)^{-i}
\end{equation}
for any $i\geq 0$, being $E_j$ the exceptional divisor of $\pi _j$. By being stable, we mean that
$$\mathrm{Sing}(\mathcal{G}_{j})=\mathrm{\underline{Max}\, mult}(X_j)$$
if $X_j$ is the strict transform (or pullback, as corresponds) of $X_{j-1}$ in $V_{j}$.

\vspace{0.2cm}

This justifies a notion of resolution of a Rees algebra. A resolution of $\mathcal{G}$ is a sequence as in (\ref{eq:seq_perm}) where the $\pi _i$ are blow ups at regular closed subsets of $\mathrm{Sing}(\mathcal{G}_{l})$, and such that $\mathrm{Sing}(\mathcal{G}_{l})=\mathrm{\underline{Max}\, mult}(X_l)=\emptyset$ (see \cite[Sections 1.1 to 1.3]{Br_E_P-E}).\\

\begin{Rem}
Given $X$ and $\xi $, there is not a unique $\mathcal{O}_{V}$-Rees algebra $\mathcal{G}$ representing the multiplicity of $X$ at $\xi $. However, it can be shown that all Rees algebras representing the maximum multiplicity of $X$ at $\xi $ are somehow equivalent: they all undergo the same resolution, and they share the same order at any point of $\mathrm{\underline{Max}\, mult}(X)$. This is the case, for instance, of the differential closure\footnote{A Rees algebra $\mathcal{G}=\oplus _{i\geq 0}I_iW^i$ is differentially closed if, for any differential opperator $D$ of order $l$, $0\leq l\leq i$, we have $D(I_i)\subset I_{i-l}$  (see \cite{V3}).} of any Rees algebra $\mathcal{G}_X$ which represents the maximum multiplicity of $X$ at $\xi $. For details about these facts, see \cite{Br_G-E_V}.
\end{Rem}

Let us suppose, for simplicity, that $X=\mathrm{Spec}(B)$ is affine. Otherwise, since we will work locally, it is enough to consider open affine subsets of $X$. It is possible to find a local \'etale immersion $X\hookrightarrow \mathrm{Spec}(R)$ into a regular scheme of dimension $n>d$ and a Rees algebra $\mathcal{G}$ over $R$, representing the multiplicity of $X$ locally in a neighborhood of $\xi $ (see \cite{V}). Under these hypotheses, we have a regular $k$-algebra $S$ of dimension $d$ and a projection $\beta :\mathrm{Spec}(R)\longrightarrow V^{(d)}=\mathrm{Spec}(S)$ inducing a finite projection
$$\beta _X:X\longrightarrow V^{(d)}=\mathrm{Spec}(S)$$
of generic rank $b$ which is also transversal for $\mathcal{G}$, that is, $\mathrm{Ker}(d\beta )$ intersects the tangent space of $\mathcal{G}$ at $\xi $ (see \cite[16.1]{Br_V2}) only at $0$, $d\beta $ being the morphism induced by $\beta $ between the tangent spaces. This projection induces a homeomorphism between $\mathrm{\underline{Max}\, mult}(X)$ and its image (see \cite[Apendix A]{Br_V2}) and an injective finite morphism of the form
$$S\longrightarrow B\cong S[x_1,\ldots ,x_{n-d}]/I(X)=S[\overline{x}_1,\ldots ,\overline{x}_{n-d}]\mbox{.}$$
We obtain in this manner a local immersion of $X$ in a smooth $n$-dimensional space $V^{(n)}=\mathrm{Spec}(R)$ in a neighborhood of $\xi $, where $R=S[x_1,\ldots ,x_{n-d}]$. There exist $f_1,\ldots ,f_{n-d}\in I(X)\subset R$ such that for certain $b_1,\ldots ,b_{n-d}\in \mathbb{Z}_{>0}$, the Rees algebra
\begin{equation}
\mathcal{G}=R[f_1W^{b_1},\ldots ,f_{n-d}W^{b_{n-d}}]
\end{equation}
represents the multiplicity of $X$ locally in a neighborhood of $\xi $, and moreover, $f_i$ is the minimal polynomial of $\overline{x}_i$ over $S$, and hence it is a monic polynomial in $x_i$ with coefficients in $S$, for $i=1,\ldots ,n-d$ (see \cite{V} for the result on the existence and construction of such a presentation, and \cite{Br_E_P-E} for notation as used here). Each $f_i$ defines a hypersurface $X_i$ in $\mathrm{Spec}(S[x_i])$. Assume that we choose the differentially closed algebra
\begin{equation}\label{eq:pres_mult}
\mathcal{G}_X^{(n)}=\mathrm{Diff}(R[f_1W^{b_1},\ldots ,f_{n-d}W^{b_{n-d}}])\mbox{,}
\end{equation}
which also represents the multiplicity of $X$ locally in a neighborhood of $\xi $. We can suppose that the maximal ideal $\mathcal{M}_{\xi }$ of $\xi $ in $R$ is given by $<x_1,\ldots ,x_{n-d},z_1,\ldots ,z_d>$ for a regular system of parameters $\left\{ z_1,\ldots ,z_d\right\} $ in $S$ (see \cite[Section 4]{Br_V2}). The image $\xi ^{(d)}$ of $\xi $ by $\beta _X$ is then defined by the maximal ideal $\mathcal{M}_{\xi ^{(d)}}=<z_1,\ldots ,z_d>$. Note that $R\longrightarrow B$ is surjective, and for any $i=1,\ldots ,n-d$ the following diagram commutes
\begin{equation}\label{diag:fact_hyp}
\xymatrix{
\mathcal{G}_X^{(n)} & R=S[x_1,\ldots ,x_{n-d}] \ar[r] & S[x_1,\ldots ,x_{n-d}]/(f_1,\ldots ,f_{n-d}) \ar[r] & B \ar[r] & 0\\ 
\mathcal{G}_{X_i}^{(d+1)} & S[x_i] \ar[u] \ar[r] & S[x_i]/(f_i) \ar[u] \ar[ur]_{\beta _{X_i}^*}& & \\
\mathcal{G}_X^{(d)}\supset \mathcal{G}_{X_i}^{(d)} & S \ar[u] \ar[ur] & & & \\
}
\end{equation}
The homomorphism $S\longrightarrow R$ happens to induce an elimination 
$$\beta :V^{(n)}=\mathrm{Spec}(R)\longrightarrow V^{(d)}=\mathrm{Spec}(S)$$
for $\mathcal{G}_X^{(n)}$ (see \cite{V}), that is, a transversal admissible projection, defining a homeomorphism between $\mathrm{Sing}(\mathcal{G}_X^{(n)})\subset V^{(n)}$ and $\beta (\mathrm{Sing}(\mathcal{G}_X^{(n)}))\subset V^{(d)}$, and such that $\mathcal{G}_X^{(d)}=\mathcal{G}_X^{(n)}\cap S$ represents the multiplicity of $\beta (X)$ (see \cite{Br_V}, \cite{Br_V1}, \cite[16 and Appendix A]{Br_V2} and \cite[Theorem 4.11 and Theorem 4.13]{V07} for properties and results on elimination). Such an elimination is equivalent to the possibility of reducing a problem of Resolution of Rees algebras in dimension $n$ to a problem of Resolution of Rees algebras in dimension $d<n$. Since $\mathcal{G}_{X}^{(n)}$ is differentially closed, by means of the elimination via $\beta $, we have the following description for $\mathcal{G}_{X}^{(n)}$ (see \cite[Example 1.5.4 and Example 1.5.15]{Br_E_P-E}):
\begin{equation}\label{eq:Gn_simple}
\mathcal{G}_X^{(n)}=S[x_1][x_1W]\odot \ldots \odot S[x_{n-d}][x_{n-d}W]\odot \mathcal{G}_X^{(d)}\mbox{,}
\end{equation}
where we have used the fact that $f_i$ is a monic polynomial in $x_i$ for $i=1,\ldots ,n-d$, and $\mathcal{G}\odot \mathcal{H}$ denotes the smallest Rees algebra containing both $\mathcal{G}$ and $\mathcal{H}$. Furthermore, one may consider, for each $i\in \left\{ 1,\ldots ,n-d \right\} $, the hypersurface $X_i\subset \mathrm{Spec}(S[x_i])$ defined by $f_i$, as it was done in \cite[Section 4.2.1]{Br_E_P-E}. For each $i$, $\mathcal{G}_{X_i}^{(d+1)}=S[x_i][f_iW^{b_i}]$ represents the maximum multiplicity of $X_i$ locally at $\beta _{X_i}(\xi )$. Then, $\mathcal{G}_X^{(n)}$ can be written as
\begin{equation}\label{eq:GcomoHypersups}
\mathcal{G}_X^{(n)}=\mathcal{G}_{X_1}^{(d+1)}\odot \ldots \odot \mathcal{G}_{X_{n-d}}^{(d+1)}\mbox{,}
\end{equation}
and
\begin{equation}\label{eq:elim_hypers}
\mathcal{G}_X^{(d)}=\mathcal{G}_{X_1}^{(d)}\odot \ldots \odot \mathcal{G}_{X_{n-d}}^{(d)}\mbox{.}
\end{equation}
Let $\varphi $ be an arc in $X$ through $\xi $ which is not contained in $\mathrm{\underline{Max}\, mult}(X)$. We may project $\varphi $ to an arc $\varphi ^{(d)}$ in $V^{(d)}$ through $\xi ^{(d)}$ via $\beta _X$, that is: $\varphi ^{(d)}=\varphi \circ \beta _X^*$. We obtain a commutative diagram
\begin{equation*}
\xymatrix@R=0.3pc@C=5pt{
\mathcal{O}_{X,\xi } \ar[rrrd]^{\varphi } & & & \\ 
 & & & K[[t]] \\
\mathcal{O}_{V^{(d)},\xi ^{(d)}} \ar[uu]^{\beta _X^*} \ar[rrru]^{\varphi ^{(d)}} & & & \\
}
\end{equation*}
In particular, note that $\varphi (\mathcal{M}_{\xi })\supset \varphi ^{(d)}(\mathcal{M}_{\xi ^{(d)}})$, so
\begin{equation}\label{eq:orders_arcs}
\mathrm{ord}(\varphi )=\mathrm{ord}_t(\varphi (\mathcal{M}_{\xi }))\leq \mathrm{ord}_t(\varphi ^{(d)}(\mathcal{M}_{\xi ^{(d)}}))=\mathrm{ord}(\varphi ^{(d)})\mbox{.}
\end{equation}
Along this paper, we will repeatedly define arcs through regular systems of parameters. For instance, to define an arc in $V^{(d)}$ through $\xi ^{(d)}$, we will do it by giving the images of a r.s.p. $\left\{ y_1,\ldots ,y_d\right\} \subset \mathcal{O}_{V^{(d)},\xi ^{(d)}}$ by the arc. The fact that this description determines the arc completely is a consequence of the continuity of the completion map $\mathcal{O}_{V^{(d)},\xi ^{(d)}}\longrightarrow \widehat{\mathcal{O}_{V^{(d)},\xi ^{(d)}}}$ of $\mathcal{O}_{V^{(d)},\xi ^{(d)}}$ at $\xi ^{(d)}$, which allows us to define $\varphi $ via a map $\widehat{\mathcal{O}_{V^{(d)},\xi ^{(d)}}}\cong K[[y_1,\ldots ,y_d]]\longrightarrow K[[t]]$. This map induces an arc in $V^{(d)}$ through $\xi ^{(d)}$, if $K$ is the residue field of $\mathcal{O}_{V^{(d)},\xi ^{(d)}}$.

\vspace{0.2cm}

We denote by $\varphi (\mathcal{G}_X^{(n)})$ the Rees algebra over $K[[t]]$ generated by the images of the $f_i$ in (\ref{eq:pres_mult}) by $\varphi $ with their respective weights. That is, 
$$\varphi (\mathcal{G}_X^{(n)})=K[[t]][\varphi (f_1)W^{b_1},\ldots ,\varphi (f_{n-d})W^{b_{n-d}}]\mbox{.}$$

Given $X$, $\xi \in \mathrm{\underline{Max}\, mult}(X)$ and an arc $\varphi $ in $X$ through $\xi $, the \textit{order of contact of $\varphi $ with $\mathrm{\underline{Max}\, mult}(X)$}, denoted by $r_{X,\varphi }$, is defined as the order of the Rees algebra $\varphi (\mathcal{G}_X^{(n)})$. The quotient $\bar{r}_{X,\varphi }=\frac{r_{X,\varphi }}{\mathrm{ord}(\varphi )}$ gives a more interesting version of this invariant because it avoids the influence of the order of the arc (see \cite[Section 3.2]{Br_E_P-E}).

\vspace{0.2cm}

In order to express $\varphi (\mathcal{G}_X^{(n)})$ by means of the decomposition in (\ref{eq:GcomoHypersups}), we may consider the projections of $\varphi $ over the $X_i$, that we shall denote by $\varphi ^{(d+1)}_i$, and which are actually arcs in the corresponding $X_i$ through $\beta _{X_i}(\xi )$, because $f_i\in I(X)$ for $i=1,\ldots ,n-d$: 
\begin{equation}\label{diag:fact_hyp_arcs}
\xymatrix@R=30pt@C=70pt{
B \ar[dr]^{\varphi } & \\ 
S[x_i]/(f_i) \ar[u]^{\beta _{X_i}^*} \ar[r]^{\varphi _i^{(d+1)}=\varphi \circ \beta _{X_i}^*} & K[[t]]  \\
S \ar[u] \ar@/^2.5pc/[uu]^{\beta _X^*} \ar[ur]_{\varphi _i^{(d)}=\varphi \circ \beta _X^*=\varphi ^{(d)}} &  \\
}
\end{equation}

The following Lemma shows how $r_{X,\varphi }$ can be computed using the expressions in (\ref{eq:Gn_simple}) and (\ref{eq:GcomoHypersups}). 
\vspace{0.5cm}
\begin{Lemma}\cite[cf. Section 4]{Br_E_P-E}\label{lemma:x_1_not_important}
Let $X$ be as in the beginning of this section, and let $\xi $ be a point in $\mathrm{\underline{Max}\, mult}(X)$. Let $\varphi $ be an arc in $X$ through $\xi $. Then
\begin{enumerate}
	\item $r_{X,\varphi }=\mathrm{ord}_{t}(\varphi ^{(d)}(\mathcal{G}_X^{(d)}))$ and
	\item $\mathrm{ord}_t(\varphi (x_i))\geq \mathrm{ord} _{t}(\varphi ^{(d)}(\mathcal{G}_X^{(d)}))$ for $i=1,\ldots ,n-d$.
\end{enumerate}
\end{Lemma}

\begin{proof}
It follows from (\ref{eq:Gn_simple}) and (\ref{eq:GcomoHypersups}) that
\begin{align*}
 & r_{X,\varphi }= \mathrm{ord}_{t}(\varphi (\mathcal{G}_X^{(n)}))=\mathrm{min}_{i=1,\ldots ,n-d} \left\{ \mathrm{ord}_{t}(\varphi _i^{(d+1)}(\mathcal{G}_{X_i}^{(d+1)})) \right\} = \\
  = \mathrm{min} & \left\{ \mathrm{ord}_t(\varphi (x_1)), \ldots ,\mathrm{ord}_t(\varphi (x_{n-d})), \mathrm{ord}_{t}(\varphi ^{(d)}(\mathcal{G}_X^{(d)})) \right\}  \leq  \mathrm{ord}_{t}(\varphi ^{(d)}(\mathcal{G}_X^{(d)}))\mbox{.}
\end{align*}
On the other hand, for each $i$, by \cite[Lemma 4.1.2]{Br_E_P-E},
\begin{equation}\label{eq:lemma_1}\mathrm{ord}_{t}(\varphi _i^{(d+1)}(\mathcal{G}_{X_i}^{(d+1)}))=\mathrm{min}\left\{ \mathrm{ord}_t(\varphi _i^{(d+1)}(x_i)),\mathrm{ord}_{t}(\varphi _i^{(d)}(\mathcal{G}_{X_i}^{(d)})) \right\} = \mathrm{ord}_{t}(\varphi _i^{(d)}(\mathcal{G}_{X_i}^{(d)}))\mbox{,}
\end{equation}
so 
$$r_{X,\varphi }=\mathrm{min}_{i=1,\ldots ,n-d}\left\{ \mathrm{ord}_{t}(\varphi _i^{(d)}(\mathcal{G}_{X_i}^{(d)})) \right\} \mbox{.}$$
But note that $\mathcal{G}_{X_i}^{(d)}\subset \mathcal{G}_{X}^{(d)}$ and $\varphi _i^{(d)}=\varphi ^{(d)}$ (see (\ref{eq:elim_hypers}) and (\ref{diag:fact_hyp_arcs})). Thus, 
$$\varphi _i^{(d)}(\mathcal{G}_{X_i}^{(d)})=\varphi ^{(d)}(\mathcal{G}_{X_i}^{(d)})\subset \varphi ^{(d)}(\mathcal{G}_{X}^{(d)})$$
and
\begin{equation}\label{eq:lemma_2}
\mathrm{ord}_t(\varphi _i^{(d)}(\mathcal{G}_{X_i}^{(d)}))\geq \mathrm{ord}_t(\varphi ^{(d)}(\mathcal{G}_{X}^{(d)}))\mbox{.}
\end{equation}
Consequently,
$$\mathrm{ord}_t(\varphi ^{(d)}(\mathcal{G}_{X}^{(d)}))\geq r_{X,\varphi }\geq \mathrm{ord}_{t}(\varphi ^{(d)}(\mathcal{G}_X^{(d)}))\mbox{,}$$
proving $1$. Now $2$ is a consequence of (\ref{eq:lemma_1}), together with (\ref{eq:lemma_2}) and the fact that, for all $i=1,\ldots ,n-d$,
$$\varphi _i^{(d)}(x_i)=\varphi _i^{(d+1)}(x_i)=\varphi (x_i)\mbox{.}$$
\end{proof}

\section{Proof of the main result}

In order to prove Theorem \ref{thm:main}, let us divide it in two one side implications, reformulated in Propositions \ref{prop:sing_small} and \ref{prop:sing_big} respectively, in a way that will be more convenient for their respective proofs. We first give a simple version of the proof of the easier one:

\begin{Prop}\label{prop:sing_small}
Let $\xi $ be an isolated point of $\mathrm{\underline{Max}\, mult}(X)$. Then there exists a positive integer $Q\in \mathbb{Z}_{>0}$, depending on $X$ and $\xi $, such that for any arc $\varphi $ in $X$ through $\xi $, 
$$\bar{r}_{X,\varphi }\leq Q\mbox{.}$$
\end{Prop}

\begin{proof}
Consider the graded structure of a Rees algebra $\mathcal{G}_X^{(n)}$ representing the multiplicity of $X$ in a neighborhood of $\xi $ as in (\ref{eq:pres_mult}),
$$\mathcal{G}_{X}^{(n)}=\oplus _{i\geq 0}I_iW^i\mbox{.}$$
Since $\mathcal{G}_X^{(n)}$ is differentially closed, the set $\mathrm{\underline{Max}\, mult}(X)$ is determined by the zeros of the ideal $I_1$ (see \cite[Proposition 4.4]{V3}). Therefore, $\mathrm{\underline{Max}\, mult}(X)$ being of dimension $0$ is equivalent to $\sqrt{I_1}$ being a maximal ideal, which, for a (any) regular system of parameters $\left\{ x_1,\ldots ,x_{n-d},z_1,\ldots ,z_d \right\} $ in $\mathcal{O}_{X,\xi }$, is also equivalent to $I_1$ containing some ideal of the form 
$$(x_1^{a_1},\ldots ,x_{n-d}^{a_{n-d}},z_1^{a_{n-d+1}},\ldots ,z_{d}^{a_{n}})$$
for some positive integers $a_1,\ldots ,a_{n}$. Note that this implies that
$$\mathcal{G}_X^{(n)}\supset \mathcal{O}_{X,\xi }[x_1^{a_1}W,\ldots ,x_{n-d}^{a_{n-d}}W,z_1^{a_{n-d+1}}W,\ldots ,z_d^{a_n}W] \mbox{.}$$
Therefore,
$$\varphi (\mathcal{G}_X^{(n)})\supset \mathcal{O}_{X,\xi }[\varphi (x_1^{a_1})W,\ldots ,\varphi (x_{n-d}^{a_{n-d}})W,\varphi (z_1^{a_{n-d+1}})W,\ldots ,\varphi (z_d^{a_n})W] \mbox{,}$$
and
\small
$$\mathrm{ord}_{t}(\varphi (\mathcal{G}_X^{(n)}))\leq \mathrm{min}\left\{ a_1\cdot \mathrm{ord}_t(\varphi (x_1)),\ldots ,a_{n-d}\cdot \mathrm{ord}_t(\varphi (x_{n-d})),a_{n-d+1}\cdot \mathrm{ord}_t(\varphi (z_1)),\ldots ,a_{n}\cdot \mathrm{ord}_t(\varphi (z_{d})) \right\} \mbox{.}$$
\normalsize
Thus
$$\bar{r}_{X,\varphi }\leq a_j\in \mathbb{Z}_{>0}$$
for any $j\in \{ 1,\ldots ,n-d\} $ such that $\mathrm{ord}(\varphi )=\mathrm{ord}_t(\varphi (x_j))$ or any $j\in \{ n-d+1,\ldots ,n \} $ such that $\mathrm{ord}(\varphi )=\mathrm{ord}_t(\varphi (z_{j-n+d}))$.
\end{proof}

\vspace{0.2cm}

The bound given by this proof is not optimal. In general, a rational number which will be smaller than the integer given by the $a_j$'s can be found, yielding an optimal bound. Note that this rational number is an invariant of $X$ at $\xi $, but since it is not needed in the proof of Theorem \ref{thm:main}, we ignore it here.

\vspace{0.2cm}

\begin{Rem}
For some arcs, we can say more about $\bar{r}_{X,\varphi }$: If $\varphi $ is such that  $\mathrm{ord}(\varphi )=\mathrm{ord}_t(\varphi (x_j))$ for some $j\in \left\{ 1,\ldots ,n-d \right\} $, then $\bar{r}_{X,\varphi }=1$. Indeed, $a_1=\ldots =a_{n-d}=1$ in the proof of Proposition \ref{prop:sing_small}, because $x_1,\ldots ,x_{n-d}\in I_1$ (see (\ref{eq:Gn_simple})).
\end{Rem}

\vspace{0.2cm}

In the next section, a precise upper bound will be given under some special condition over $X$ at $\xi $, in terms of orders of elimination algebras. This condition is related with the $\tau $ invariant of $\mathcal{G}_X^{(n)}$ at $\xi $.

\vspace{0.2cm}

We prove now the most delicate implication. To make the proof easier to understand, we will deal separately with an easy case first, even though it of course follows from the general one, which we prove afterwards. The reader unfamiliar with the techniques of resolution used in this proof, as well as definitions of strict and total transform of an ideal, can consult them in \cite[Section 7]{Br_E_V} or \cite{E_V}.

\vspace{0.2cm}

\begin{Prop}\label{prop:sing_big}
If $\xi $ lies in a component of $\mathrm{\underline{Max}\, mult}(X)$ of dimension greater or equal to $1$, then for any $q\in \mathbb{Q}$, one can find an arc $\varphi $ in $X$ through $\xi $ such that 
$$\bar{r}_{X,\varphi }>q\mbox{.}$$
\end{Prop}

\begin{proof}
Since $r_{X,\varphi }=\mathrm{ord}_{t}(\varphi ^{(d)}(\mathcal{G}_X^{(d)}))$ if $\varphi ^{(d)}=\varphi \circ \beta _X^*$ (see Lemma \ref{lemma:x_1_not_important}), our strategy here will be choosing an arc $\bar{\varphi }^{(d)}$ in $V^{(d)}$ through $\xi ^{(d)}$ which gives $\mathrm{ord}_{t}(\bar{\varphi }^{(d)}(\mathcal{G}_X^{(d)}))$ big enough first, and then lifting it via $\beta _X$ to an arc $\varphi $ in $X$ through $\xi $, proving afterwards that it satisfies the statement in the Proposition.

\vspace{0.2cm}

Suppose first that there exists a smooth curve $\widetilde{C}\subset \mathrm{\underline{Max}\, mult}(X)$ containing $\xi $. Then $C=\beta _X(\widetilde{C})\subset V^{(d)}$ is a smooth curve containing $\xi ^{(d)}$(see \cite[Theorem 6.3]{V}). Assume that $C$ is defined by a prime ideal $J\subset \mathcal{O}_{V^{(d)},\xi ^{(d)}}$. Consider the family of arcs $\bar{\varphi }_N^{(d)}$ in $V^{(d)}$ through $\xi ^{(d)}$, for $N\in \mathbb{Z}_{>0}$, given by
\begin{align*}
\bar{\varphi }_N^{(d)}: \mathcal{O}_{V^{(d)},\xi ^{(d)}} & \longrightarrow K[[t]]\mbox{,}\\
J & \longmapsto t^N\mbox{,}\\
\mathcal{M}_{\xi ^{(d)}} & \longmapsto t\mbox{.}
\end{align*}
This can be done because we may assume that, in this situation, $J=(y_2,\ldots ,y_{d})$ for some regular system of parameters $\left\{ y_1,\ldots ,y_d \right\} $ of $\mathcal{O}_{V^{(d)},\xi ^{(d)}}$. Then, such a family of arcs could be constructed by just defining $\bar{\varphi }_N^{(d)}(y_1)=t$ and $\bar{\varphi }_N^{(d)}(y_j)=t^N$ for $j=2,\ldots ,d$. For any $N\in \mathbb{N}$, the arc $\bar{\varphi }_N^{(d)}$ can be lifted to an arc $\varphi _N$ in $X$ through $\xi $ satisfying $\bar{r}_{X,\varphi _N}\geq N$ as follows:

\vspace{0.2cm}

Note that we are under the hypothesis $d\geq 2$. Consider the ideal $\mathcal{P}=\mathrm{Ker}(\bar{\varphi }_N^{(d)})\subset \mathcal{O}_{V^{(d)},\xi ^{(d)}}$. There exists a prime ideal $\mathcal{Q}$ in $\mathcal{O}_{X,\xi }$ dominating $\mathcal{P}$. We have the following commutative diagram:
$$
\xymatrix{ 
\mathcal{Q} \subset \mathcal{O}_{X,\xi } \ar[r]^{\mu } & \mathcal{O}_{X,\xi }/\mathcal{Q} \\
\mathcal{P} \subset \mathcal{O}_{V^{(d)},\xi ^{(d)}} \ar[r]^{\mu ^{(d)}} \ar@<-2ex>[u]^{\beta _X^*} & \mathcal{O}_{V^{(d)},\xi ^{(d)}}/\mathcal{P} \ar[u] ^{\bar{\beta _X}^*}
}
$$
where the vertical arrows are finite morphisms, and both rings on the right side are $1$-dimensional, so $\mathcal{Q}$ defines a curve. One can find a nontrivial arc $\tilde{\varphi }_N: \mathcal{O}_{X,\xi }/\mathcal{Q} \longrightarrow K[[t]]$ through $\mu (\xi )$,\footnote{Here $K$ will be the residue field of $\mathcal{O}_{X,\xi }/\mathcal{Q}$ at $\mu (\xi )$.} which induces also an arc 
$$\varphi _N=\tilde{\varphi _N}\circ \mu :\mathcal{O}_{X,\xi }\longrightarrow K[[t]]$$
through $\xi $, and
$$\varphi ^{(d)}_N=\varphi _N \circ \beta _X^*=\tilde{\varphi }_N \circ \bar{\beta _X}^* \circ \mu ^{(d)}: \mathcal{O}_{V^{(d)},\xi ^{(d)}} \longrightarrow K[[t]]$$
with
$$\mathrm{Ker}(\varphi ^{(d)}_N)=\mathcal{P}=\mathrm{Ker}(\bar{\varphi }^{(d)}_N)=(y_2-y_1^N, y_2-y_j:2<j\leq d)\subset \mathcal{O}_{V^{(d)},\xi ^{(d)}}\mbox{.}$$
Since $C\subset \mathrm{Sing}(\mathcal{G}_X^{(d)})$, 
$$\mathrm{ord}_{C}(I_i\mathcal{O}_{V^{(d)},\xi ^{(d)}})\geq i\mbox{\; }\forall  i\geq 0\mbox{,}$$
so $I_i\mathcal{O}_{V^{(d)},C}\subset J^i\mathcal{O}_{V^{(d)},C}$. But note that $J$ is a regular prime in $\mathcal{O}_{V^{(d)},\xi ^{(d)}}$ defining $C$, so $I_i\mathcal{O}_{V^{(d)},\xi ^{(d)}}\subset J^i$ for all $i\geq 0$. Consequently,
$$\mathcal{G}_X^{(d)}\subset \mathcal{O}_{V^{(d)},\xi ^{(d)}}[JW]\mbox{,}$$
and
$$\varphi _N^{(d)}(\mathcal{G}_X^{(d)})\subset \varphi _N^{(d)}(\mathcal{O}_{V^{(d)},\xi ^{(d)}}[JW])\mbox{.}$$
Hence, for $\varphi _N^{(d)}$ constructed as above, 
$$\mathrm{ord}_{t}(\varphi _N^{(d)}(\mathcal{G}_X^{(d)}))\geq \mathrm{ord}_{t}(\varphi _N^{(d)}(\mathcal{O}_{V^{(d)},\xi ^{(d)}}[JW]))=\mathrm{ord}_t(\varphi _N^{(d)}(J))\mbox{.}$$
Using also Lemma \ref{lemma:x_1_not_important} and the fact that $\mathrm{ord}(\varphi _N)\leq \mathrm{ord}(\varphi _N^{(d)})=\mathrm{ord}_{t}(\varphi _N^{(d)}(\mathcal{M}_{\xi ^{(d)}}))$ (see (\ref{eq:orders_arcs})), we arrive to
$$\bar{r}_{X,\varphi _N}=\frac{\mathrm{ord}_{t}(\varphi _N^{(d)}(\mathcal{G}_X^{(d)}))}{\mathrm{ord}(\varphi _N)}\geq \frac{\mathrm{ord}_t(\varphi _N^{(d)}(J))}{\mathrm{ord}(\varphi _N)}\geq \frac{\mathrm{ord}_t(\varphi _N^{(d)}(J))}{\mathrm{ord}(\varphi _N^{(d)})}=\frac{\mathrm{ord}_t(\varphi _N^{(d)}(J))}{\mathrm{ord}_{t}(\varphi _N^{(d)}(\mathcal{M}_{\xi ^{(d)}}))}\mbox{.}$$
Assume that $\varphi _N^{(d)}(y_j)=u_jt^{\alpha _j}$ for $j=1,\ldots ,d$ for some $u_j$ units in $K[[t]]$ and some $\alpha _j\in \mathbb{Z}_{>0}$. Then 
$$\varphi _N^{(d)}(y_2-y_1^N)=0=\varphi _N^{(d)}(y_2)-\varphi _N^{(d)}(y_1)^N=u_2t^{\alpha _2}-u_1^Nt^{\alpha _1\cdot N}\mbox{\;  and}$$
$$\varphi _N^{(d)}(y_2-y_j)=0=\varphi _N^{(d)}(y_2)-\varphi _N^{(d)}(y_j)=u_2t^{\alpha _2}-u_jt^{\alpha _j}\mbox{\; for\; }2<j\leq d\mbox{.}$$
Necessarily
$$\alpha _2=\alpha _1\cdot N\mbox{\; and}$$
$$\alpha _2=\alpha _j\mbox{\; for\; }2<j\leq d\mbox{,}$$
so
$$\bar{r}_{X,\varphi _N}\geq \frac{\mathrm{ord}_t(\varphi _N^{(d)}(J))}{\mathrm{ord}_{t}(\varphi _N^{(d)}(\mathcal{M}_{\xi ^{(d)}}))}=\frac{\mathrm{min}_{i=2,\ldots ,d}\left\{ \alpha _i \right\}}{\mathrm{min}_{j=1,\ldots ,d}\left\{ \alpha _j \right\}}=\frac{\alpha _2}{\alpha _1}=N$$
which, for a fixed $q\in \mathbb{Q}$, can be greater than $q$ by just choosing $N$ big enough.

\vspace{0.2cm}

Suppose now that $\widetilde{C}\subset \mathrm{\underline{Max}\; mult}(X)$ is not smooth.  As before, assume that $C=\beta(\widetilde{C})=V(J)\subset V^{(d)}$ for some ideal $J\subset \mathcal{O}_{V^{(d)},\xi ^{(d)}}$. Consider the following sequence:
\begin{equation}\label{diag:strong_desing}
\xymatrix@R=0.3pc@C=5pt{
V^{(d)} & = & V_0^{(d)} & & & V_1^{(d)} \ar[lll]_{\pi _1} & & & \ldots \ar[lll]_{\pi _2} & & & V_r^{(d)} \ar[lll]_{\pi _r}\\
& & \cup & & & \cup & & & \ldots & & & \cup \\
C & = & C_0 & & & C'_1 & & & \ldots & & & C'_r\\
\xi ^{(d)} & = & \xi _0^{(d)} & & & \xi _1^{(d)} & & & \ldots & & & \xi _r^{(d)}
}
\end{equation}
where $\pi _i$ is the blow up at the point $\xi _{i-1}^{(d)}$, and $\xi _i^{(d)}\in \pi _i^{-1}(\xi _{i-1}^{(d)})\cap C_i'$ for $i=1,\ldots ,r$, and such that the strict transform $C'_r$  of $C_0$ by $\pi =\pi _1\circ \ldots \circ \pi _r$ is a smooth curve having normal crossings with the exceptional divisor at $\xi _r^{(d)}$. Such a sequence can always be found, being an embedded desingularization of $C$. Let us look now at the total transform $J_r=J\mathcal{O}_{V^{(d)}_r}$ of the ideal $J$ by $\pi $, which will be, locally in a neighborhood of $\xi _r^{(d)}$, of the form
$$J_r=\mathscr{M}\cdot J_r'\mbox{,}$$
where $J'_r$ is contained in the the ideal $I(C_r')$ defining the strict transform $C'_r$ of $C$ in $V_r^{(d)}$, and $\mathscr{M}$ is a locally a monomial. Let us choose a family of arcs $\bar{\varphi }_{N,r}^{(d)}$ in $V_r^{(d)}$ through $\xi _r^{(d)}$ for $N\in \mathbb{Z}_{>0}$ such that $\bar{\varphi }_{N,r}^{(d)}(I(C'_r))=t^N$ and $\bar{\varphi }_{N,r}^{(d)}(\pi ^*(\mathcal{M}_{\xi ^{(d)}}))=t^a$ for some $a \in \mathbb{Z}_{>0}$ constant, as we did for the case of $C$ smooth. For this, note that locally in a neighborhood of $\xi _r^{(d)}$, one can consider a regular system of parameters in $\mathcal{O}_{V^{(d)}_r,\xi _r^{(d)}}$ given by 
$$\left\{ \tilde{y}_1=I(H_1),\tilde{y}_2,\ldots ,\tilde{y}_{d} \right\} \mbox{,}$$
so that $I(C_r')=(\tilde{y}_2,\ldots ,\tilde{y}_{d})$, and moreover
$$\pi ^*(\mathcal{M}_{\xi ^{(d)}})=I(H_1)^{a}$$
for $a\in \mathbb{N}$, where $H_1=\pi _r^{-1}(\xi _{r-1})$ is the exceptional divisor of $\pi _r$, because of the way in which the centers of the $\pi _i$ are chosen. Consider $\bar{\varphi }_{N,r}^{(d)}$ given as
\begin{align*}
\bar{\varphi }_{N,r}^{(d)}: \mathcal{O}_{V_r^{(d)},\xi _r^{(d)}} & \longrightarrow K[[t]]\mbox{,}\\
\tilde{y}_1 & \longmapsto t\mbox{,}\\
\tilde{y}_j & \longmapsto t^N,\; \mbox{for}\; j=2,\ldots ,d\mbox{,}
\end{align*}
which satisfies the desired properties. Note that $\pi $ induces a sequence of permissible transformations of $X$ via $\beta _X$:
\begin{equation*}
\xymatrix@R=1.2pc@C=5pt{
X \ar[d]^{\beta _X} & = & X_0 & & & & & & & & & X_r \ar[lllllllll]_{\pi _X} \ar[d]^{\beta _{X_r}}\\
V^{(d)} & = & V_0^{(d)} & & & V_1^{(d)} \ar[lll]_{\pi _1} & & & \ldots \ar[lll]_{\pi _2} & & & V_r^{(d)} \ar[lll]_{\pi _r}\\
}
\end{equation*}
For each $N\in \mathbb{Z}_{>0}$, $\bar{\varphi }_{N,r}^{(d)}$ can be lifted to an arc in $X_r$ through $\xi _r^{(d)}$ via a diagram as in the regular case:
$$
\xymatrix{ 
\mathcal{Q} \subset \mathcal{O}_{X_r,\xi _r} \ar[r]^{\mu } & \mathcal{O}_{X_r,\xi _r}/\mathcal{Q} \\
\mathcal{P} \subset \mathcal{O}_{V^{(d)}_r,\xi _r^{(d)}} \ar[r]^{\mu ^{(d)}} \ar@<-2ex>[u]^{\beta _{X_r}^*} & \mathcal{O}_{V^{(d)}_r,\xi _r^{(d)}}/\mathcal{P} \ar[u] ^{\bar{\beta }_{X_r}^*}
}
$$
where $\mathcal{P}=\mathrm{Ker}(\bar{\varphi }_{N,r}^{(d)})=\mathcal{Q}\cap \mathcal{O}_{V_r^{(d)},\xi ^{(d)}_r}$. As we did in the case of $C$ a regular curve, we pick an arc\footnote{Now $K$ is the residue field of  $\mathcal{O}_{X_r,\xi _r}/\mathcal{Q}$ at $\mu (\xi _r)$.}
$$\tilde{\varphi }_{N,r}:\mathcal{O}_{X_r,\xi _r}/\mathcal{Q} \longrightarrow K[[t]]$$
and obtain
$$\varphi _{N,r}=\tilde{\varphi }_{N,r} \circ \mu :\mathcal{O}_{X_r,\xi _r}\longrightarrow K[[t]]\mbox{,}$$
so that $\mathrm{Ker}(\bar{\varphi }_{N,r}^{(d)})=\mathrm{Ker}(\varphi _{N,r}^{(d)})$, where
$$\varphi _{N,r}^{(d)}=\varphi _{N,r}\circ \beta _{X_r}^*:\mathcal{O}_{V^{(d)}_r,\xi _r^{(d)}} \longrightarrow K[[t]]\mbox{.}$$
Note that $\mathrm{Ker}(\varphi _{N,r}^{(d)})=(\tilde{y}_2-\tilde{y}_1^N,\tilde{y}_2-\tilde{y}_j:2<j\leq d)$, so 
$$\mathrm{ord}_t(\varphi _{N,r}^{(d)} (\tilde{y}_2))=\mathrm{ord}_t(\varphi _{N,r}^{(d)} (\tilde{y}_j))=N\cdot \mathrm{ord}_t(\varphi _{N,r}^{(d)} (\tilde{y}_1))$$
for $2<j\leq d$, and that 
$$\mathrm{ord}(\varphi _{N,r}^{(d)})=\mathrm{ord}_{t}(\varphi _{N,r}^{(d)}(\pi ^*(\mathcal{M}_{\xi ^{(d)}})))=\mathrm{ord}_t(\varphi _{N,r}^{(d)}(\tilde{y}_1^{a}))=a\cdot \mathrm{ord}_t(\varphi _{N,r}^{(d)}(\tilde{y}_1))\mbox{,}$$
so necessarily
\begin{equation}\label{eq:calculo_cociente_noliso}
\frac{\mathrm{ord}_{t}(\varphi _{N,r}^{(d)}(I(C_r')))}{\mathrm{ord}_{t}(\varphi _{N,r}^{(d)}(\pi ^*(\mathcal{M}_{\xi ^{(d)}})))}=\frac{\mathrm{min}_{j=2,\ldots ,d}\left\{ \mathrm{ord} _t(\varphi _{N,r}^{(d)}(\tilde{y}_j)) \right\}}{\mathrm{min}_{i=1,\ldots ,d}\left\{ \mathrm{ord} _t(\varphi _{N,r}^{(d)}(\tilde{y}_i)) \right\}}=\frac{N}{a}\mbox{.}
\end{equation}
Finally, we obtain
$$\varphi _N:\mathcal{O}_{X,\xi }\longrightarrow K[[t]]$$
by composing $\varphi _{N,r} \circ \pi _X^*$, and we also obtain its projection to $V^{(d)}$ as $\varphi ^{(d)}_N=\varphi _{N,r}^{(d)}\circ \pi ^*$. Note that the sequence of transformations in \ref{diag:strong_desing}) is such that the multiplicity of $X_i$ along the curve does not decrease along the process, and hence $C'_i\subset \beta _{X_i}(\mathrm{\underline{Max}\; mult}(X_i))$ for $i=0,\ldots ,r$. As a consequence, it induces a sequence of permissible transformations of Rees algebras for $\mathcal{G}_X^{(d)}$ as in 
\cite[Definition 6.1]{V07}, since for all $i=1,\ldots, r$, $\pi _i$ is a blow up at a regular closed subset of $\mathrm{Sing}(\mathcal{G}_{X,i-1}^{(d)})$:
\begin{equation}
\xymatrix@R=0.3pc@C=5pt{
V^{(d)} & = & V_0^{(d)} & & & V_1^{(d)} \ar[lll]_{\pi _1} & & & \ldots \ar[lll]_{\pi _2} & & & V_r^{(d)} \ar[lll]_{\pi _r}\\
\mathcal{G}_X^{(d)} & = & \mathcal{G}_{X,0}^{(d)}=\oplus_{i\geq 0} I_iW^i & & & \mathcal{G}_{X,1}^{(d)} \ar[lll] & & & \ldots \ar[lll] & & & \mathcal{G}_{X,r}^{(d)}=\oplus _{i\geq 0}I_{i,r}W^i \ar[lll] \\
}
\end{equation}
where 
$$I_i\mathcal{O}_{V^{(d)}_r} \subset I_{i,r}$$
for $i\geq 0$ (see (\ref{eq:trans_law_RA})). In particular,
$$\mathcal{G}_X^{(d)}\mathcal{O}_{V^{(d)}_r}=\oplus _{i\geq 0}(I_i\mathcal{O}_{V^{(d)}_r})W^i \subset \oplus _{i\geq 0}I_{i,r}W^i\mbox{.}$$
Moreover,
$$\varphi _N^{(d)}(\mathcal{G}_X^{(d)})=\varphi _{N,r}^{(d)}(\oplus _{i\geq 0}(I_i\mathcal{O}_{V^{(d)}_r})W^i)\subset \varphi _{N,r}^{(d)}(\mathcal{G}_{X,r}^{(d)})\mbox{,}$$
so
$$\mathrm{ord}_{t} (\varphi _N^{(d)}(\mathcal{G}_X^{(d)}))\geq \mathrm{ord}_{t}(\varphi _{N,r}^{(d)}(\mathcal{G}_{X,r}^{(d)}))\mbox{.}$$
Since $I(C_r')$ is a regular prime in $\mathcal{O}_{V_r^{(d)},\xi _r^{(d)}}$ defining a curve contained in $\mathrm{Sing}(\mathcal{G}_{X,r}^{(d)})$,
$$\mathcal{G}_{X,r}^{(d)} \subset \mathcal{O}_{V_r^{(d)},\xi _r^{(d)}}[I(C_r')W]\mbox{,}$$
and hence 
\begin{equation}\label{eq:ineq_ords_C_lisa}
\mathrm{ord}_{t} (\varphi _N^{(d)}(\mathcal{G}_X^{(d)}))\geq \mathrm{ord}_{t}(\varphi _{N,r}^{(d)}(\mathcal{G}_{X,r}^{(d)}))\geq \mathrm{ord}_{t}(\varphi _{N,r}^{(d)}(I(C_r')))\mbox{.}
\end{equation}

On the other hand,
\begin{equation*}
\mathrm{ord}(\varphi _N)=\mathrm{ord}_{t}(\varphi _N(\mathcal{M}_{\xi }))\leq \mathrm{ord}_t(\varphi ^{(d)}_N(\mathcal{M}_{\xi ^{(d)}}))=\mathrm{ord}(\varphi _N^{(d)})=\mathrm{ord}_t(\varphi _{N,r}^{(d)}(\pi ^*(\mathcal{M}_{\xi ^{(d)}})))\mbox{.}
\end{equation*}
This, together with Lemma \ref{lemma:x_1_not_important}, (\ref{eq:calculo_cociente_noliso}), and (\ref{eq:ineq_ords_C_lisa}) implies, for each $N\in \mathbb{Z}_{>0}$,
$$\bar{r}_{X,\varphi _N}=\frac{\mathrm{ord}_{t}(\varphi _N(\mathcal{G}_X^{(n)}))}{\mathrm{ord}(\varphi _N)}\geq \frac{\mathrm{ord}_{t}(\varphi _N^{(d)}(\mathcal{G}_X^{(d)}))}{\mathrm{ord}(\varphi _N^{(d)})}\geq \frac{\mathrm{ord}_{t}(\varphi _{N,r}^{(d)}(I(C_r')))}{\mathrm{ord}_t(\varphi _{N,r}^{(d)}(\pi ^*(\mathcal{M}_{\xi ^{(d)}})))}=\frac{N}{a}\mbox{.}$$
Again, it is clear that for a fixed $q\in \mathbb{Q}$, we may choose $N$ such that $\bar{r}_{X,\varphi _N}>q$.
\end{proof}

\vspace{0.2cm}

As was stated in the introduction, our main result means, in terms of the Nash multiplicity sequence, that $\xi $ is an isolated point of $\mathrm{\underline{Max}\, mult}(X)$ if and only if there exists an upper bound for the number of blowups as in (\ref{eq:Nms}) needed before the Nash multiplicity sequence decreases for the first time (normalized by the order of $\varphi $), for any arc $\varphi $ in $X$ through $\xi $:

\begin{Cor}\label{cor:persistance}
Let $X$ be a variety over a field $k$ of characteristic zero. A point $\xi \in \mathrm{\underline{Max}\, mult}(X)$ is an isolated point of $\mathrm{\underline{Max}\, mult}(X)$ if and only if $\mathrm{sup} _{\varphi }\left\{ \frac{\rho _{X,\varphi }}{\mathrm{ord}(\varphi )} \right\} < \infty $, where the supremum is taken over all arcs $\varphi $ in $X$ through $\xi $.
\end{Cor}

\begin{proof}
The direct implication follows from \cite[Corollary 4.3.3-1]{Br_E_P-E}. For the reverse one, assume that $\mathrm{sup} _{\varphi }\left\{ \frac{\rho _{X,\varphi }}{\mathrm{ord}(\varphi )} \right\} =q\in \mathbb{Q}_{>0}$ and get to a contradiction: for $N=\left[  q \right] +1>q$, choose $\varphi _{aN}$ as in the proof of Proposition \ref{prop:sing_big}, so that it satisfies $\bar{r}_{X,\varphi _{aN}}\geq N$. This implies 
$$\rho _{X,\varphi _{aN}}=\left[ r_{X,\varphi _{aN}} \right] \geq \left[ N\cdot \mathrm{ord}(\varphi _{aN}) \right] =N\cdot \mathrm{ord}(\varphi _{aN})\mbox{.}$$
But this is equivalent to 
$$\frac{\rho _{X,\varphi _{aN}}}{\mathrm{ord}(\varphi _{aN})}\geq \frac{ N\cdot \mathrm{ord}(\varphi _{aN}) }{\mathrm{ord}(\varphi _{aN})}=N>q\mbox{,}$$
yielding a contradiction.
\end{proof}

\section{Consequences and examples}

Assume now that $\tau _{\mathcal{G}_X^{(n)},\xi }=n-1$. Recall that $\tau _{\mathcal{G}_X^{(n)},\xi }$ is the codimension of the largest linear subspace such that the addition of this subspace with the tangent cone\footnote{The tangent cone of $\mathcal{G}_X^{(n)}=\oplus _{i\geq 0}I_iW^i$ at $\xi $ is the subspace of the tangent space of $V^{(n)}$ at $\xi $ defined by the homogeneous ideal $\oplus _{i\geq 0}I_i\cdot (\mathcal{M}_{\xi }^{i}/\mathcal{M}_{\xi }^{i+1})$, where $\mathcal{M}_{\xi }$ is the maximal ideal of $V^{(n)}$ at $\xi $.} of $\mathcal{G}_X^{(n)}$ at $\xi $ lies in the tangent cone again (see \cite[Section 4]{B} for details). Then, for some regular system of parameters $\left\{ x'_1,\ldots ,x'_{n-1},z\right\} \subset R=\mathcal{O}_{V^{(n)},\xi }$ for $\mathcal{G}_X^{(n)}$ differentially closed representing the multiplicity of $X$ at $\xi $, we have
$$x'_1W,\ldots ,x'_{n-1}W\subset \mathcal{G}_X^{(n)}\mbox{,}$$
and one can find an elimination map $V^{(n)}\stackrel{\beta _X^{(1)}}{\longrightarrow} V^{(1)}$ (see \cite[Sections 13.3 and 16.1]{Br_V2}). This means, that finding a resolution of the algebra $\mathcal{G}_X^{(n)}$ is equivalent to finding a resolution of an algebra $\mathcal{G}_X^{(1)}$ over a smooth scheme of dimension $1$, namely $V^{(1)}$. We may assume that, up to an \'etale extension, $R=S'[x'_1,\ldots ,x'_{n-1}]$, where $S'$ is a regular ring of dimension $1$. Then
\begin{equation}\label{eq:Gn_tau_grande}
\mathcal{G}_X^{(n)}=R[x'_1W]\odot \ldots \odot R[x'_{n-1}W]\odot \mathcal{G}_X^{(1)}
\end{equation}
where $\mathcal{G}_X^{(1)}\subset S'[W]$. Note that, in this situation:
$$\mathrm{ord}_{\xi }(\mathcal{G}^{(n)}_X)=1=\mathrm{ord}_{\xi ^{(n-1)}}(\mathcal{G}^{(n-1)}_X)=\ldots =\mathrm{ord}_{\xi ^{(2)}}(\mathcal{G}^{(2)}_X)< \mathrm{ord}_{\xi ^{(1)}}(\mathcal{G}^{(1)}_X)\mbox{,}$$
so $\mathrm{ord}_{\xi ^{(1)}}(\mathcal{G}^{(1)}_X)$ is the first interesting resolution invariant in this case.

\vspace{0.2cm}

Under these hypotheses $\xi $ is an isolated point of $\mathrm{\underline{Max}\, mult}(X)$, and hence Proposition \ref{prop:sing_small} guarantees that $\Phi _{X,\xi }$ is upper bounded. It turns out that the additional condition on $\tau _{\mathcal{G}_X^{(n)},\xi }$ yields an improvement of that result:

\begin{Prop}\label{thm:upper_bound_bigtau} If $\tau _{\mathcal{G}_X^{(n)},\xi }=n-1$, then for any arc $\varphi $ in $X$ through $\xi $:
$$\bar{r}_{X,\varphi }\leq \mathrm{ord}_{\xi ^{(1)}}(\mathcal{G}_X^{(1)})\mbox{,}$$
and this bound is sharp.
\end{Prop}

\begin{proof}
We may assume first that $\mathrm{min}_{i=1,\ldots ,n-1}\left\{ \mathrm{ord}_t(\varphi (x'_i)) \right\} =\mathrm{ord}_t(\varphi (x'_1))$. By (\ref{eq:Gn_tau_grande}), we obtain
\begin{equation}
r_{X,\varphi }\leq \mathrm{min}\left\{ \mathrm{ord}_t(\varphi (x'_1)), \mathrm{ord}_{t}(\varphi ^{(1)}(\mathcal{G}^{(1)}_X)) \right\} \mbox{,}
\end{equation}
where $\varphi ^{(1)}$ is the projection of $\varphi $ via the elimination map $\beta _X^{(1)}: \mathrm{Spec}(R) \longrightarrow \mathrm{Spec}(S')$. Note that, either $\mathrm{ord} (\varphi )=\mathrm{ord}_t(\varphi (x'_1))$ or $\mathrm{ord} (\varphi )=\mathrm{ord}_t(\varphi (z))$. In the first case, 
$$1\leq \bar{r}_{X,\varphi }\leq \frac{\mathrm{min}\left\{ \mathrm{ord}_t(\varphi (x'_1)), \mathrm{ord}_{t}(\varphi ^{(1)}(\mathcal{G}^{(1)}_X))  \right\}}{\mathrm{ord}_t(\varphi (x'_1))}\leq 1\mbox{,}$$
which implies that 
$$ \bar{r}_{X,\varphi }=1<\mathrm{ord}_{\xi ^{(1)}}(\mathcal{G}_X^{(1)})$$
In the second case,
$$\bar{r}_{X,\varphi }\leq \frac{\mathrm{ord}_{t}(\varphi ^{(1)}(\mathcal{G}^{(1)}_X))}{\mathrm{ord}_t(\varphi (z))}\mbox{.}$$
Note that $\mathrm{ord}_{t}(\varphi ^{(1)}(\mathcal{G}_{X}^{(1)}))\geq \mathrm{ord}_{\xi ^{(1)}}(\mathcal{G}_{X}^{(1)})\cdot \mathrm{ord}_t(\varphi (z))$ (see \cite[Lemma 4.1.6]{Br_E_P-E}). But actually this inequiality is an equality here. This follows from the fact that $\mathcal{G}_X^{(1)}\subset S'[W]$ so, for all $gW^l\in \mathcal{G}_X^{(1)}$, we have that $\mathrm{ord}_t(\varphi (g))=\mathrm{ord}_z(g)\cdot \mathrm{ord}_t(\varphi (z))$. One only needs to observe now that $\varphi ^{(1)}(\mathcal{G}_{X}^{(1)})=K[[t]][\varphi (g)W^{l}: gW^l\in \mathcal{G}_X^{(1)}]$, and the equality is clear. Hence
$$\bar{r}_{X,\varphi }\leq \frac{\mathrm{ord}_{\xi ^{(1)}}(\mathcal{G}_{X}^{(1)})\cdot \mathrm{ord}_t(\varphi (z))}{\mathrm{ord}_t(\varphi (z))} =\mathrm{ord}_{\xi ^{(1)}}(\mathcal{G}_{X}^{(1)}) \mbox{.}$$

To see that $\mathrm{ord}_{\xi ^{(1)}}(\mathcal{G}_X^{(1)})$ is a sharp bound, consider an r.s.p. $\left\{ x_1,\ldots ,x_{n-d},z_1,\ldots ,z_{d-1},z_d\right\} \subset R$ as in Section \ref{sec:notation}. Since $\tau _{\mathcal{G}_X^{(n)},\xi }=n-1$, we may assume that $x_1W,\ldots ,x_{n-d}W,z_1W,\ldots ,z_{d-1}W\in \mathcal{G}_{X}^{(n)}$. We may choose an arc $\bar{\varphi }^{(d)}$ in $V^{(d)}$ through $\beta _X(\xi )$ such that  $\bar{\varphi }^{(d)}(z_d)=t$ and $\bar{\varphi }^{(d)}(z_1)=\ldots =\bar{\varphi }^{(d)}(z_{d-1})=t^a$, for some $a\in \mathbb{Z}_{>0}$, $a> \mathrm{ord}_{\xi ^{(1)}}(\mathcal{G}_X^{(1)})>1$. This arc can be lifted to an arc $\varphi $ in $X$ through $\xi $, for which
$$\bar{r}_{X,\varphi }=\frac{\mathrm{ord}_{t}(\varphi ^{(d)}(\mathcal{G}_X^{(d)}))}{\mathrm{ord}(\varphi )}\geq \frac{\mathrm{ord}_{t}(\varphi ^{(d)}(\mathcal{G}_X^{(d)}))}{\mathrm{ord}(\varphi ^{(d)})}=\frac{\mathrm{min}\left\{ \mathrm{ord}_t(\varphi ^{(d)}(z_1)),\ldots ,\mathrm{ord}_t(\varphi ^{(d)}(z_{d-1})),\mathrm{ord}_{t}(\varphi ^{(1)}(\mathcal{G}_X^{(1)})) \right\} }{\mathrm{ord}(\varphi ^{(d)})}$$
by Lemma \ref{lemma:x_1_not_important} and  (\ref{eq:orders_arcs}), where $\varphi ^{(d)}=\varphi \circ \beta _X^*$ and $\varphi ^{(1)}=\varphi \circ (\beta _X^{(1)})^*$. Also, $\mathrm{Ker}(\varphi ^{(d)})=\mathrm{Ker}(\bar{\varphi }^{(d)})=(z_d^a-z_1,\ldots ,z_d^a-z_{d-1})$, so for $i=1,\ldots ,d-1$ it is clear that  
$$\mathrm{ord}_t(\varphi ^{(d)}(z_i))=a\cdot \mathrm{ord}_t(\varphi ^{(d)}(z_d))> \mathrm{ord}_{\xi ^{(1)}}(\mathcal{G}_X^{(1)})\cdot \mathrm{ord}_t(\varphi ^{(d)}(z_d))=\mathrm{ord}_{t}(\varphi ^{(1)}(\mathcal{G}_X^{(1)}))>\mathrm{ord}_t(\varphi ^{(d)}(z_d))\mbox{.}$$
Thus,
$$\mathrm{ord}_{\xi ^{(1)}}(\mathcal{G}_{X}^{(1)})\geq \bar{r}_{X,\varphi }\geq  \frac{\mathrm{ord}_{t}(\varphi ^{(1)}(\mathcal{G}_X^{(1)}))}{\mathrm{ord}_t(\varphi ^{(d)}(z_d))}=\mathrm{ord}_{\xi ^{(1)}}(\mathcal{G}_X^{(1)})\mbox{.}$$

\end{proof}

\vspace{0.2cm}

However, under the hypothesis of Proposition \ref{thm:upper_bound_bigtau}, sometimes it is possible to find arcs such that $\mathrm{ord}_{\xi }(\mathcal{G}_X^{(d)})=1<\bar{r}_{X,\varphi }<\mathrm{ord}_{\xi }(\mathcal{G}_{X}^{(1)})$. Let us show an example for this:

\vspace{0.2cm}

\begin{Ex}
Consider $X\hookrightarrow \mathrm{Spec}(k[x,y,z])$ defined by the equation $f=xy-z^5$ and $\xi =(0,0,0)=\mathrm{\underline{Max}\, mult}(X)$, and let $\varphi $ be the arc defined by $\varphi (x)=t^3$, $\varphi (y)=t^2$, $\varphi (z)=t$. Here
$$\mathcal{G}_X^{(3)}=\mathrm{Diff}(k[x,y,z][fW^2])=k[x][xW]\odot \mathcal{G}^{(2)}_X=k[x,y][xW,yW]\odot \mathcal{G}_X^{(1)}\mbox{,}$$
where $\mathcal{G}_X^{(2)}=k[y,z][yW,z^5W^2,z^4W]$ and $\mathcal{G}_X^{(1)}=k[z][z^5W^2,z^4W]$, so $\mathrm{ord}_{\xi ^{(d)}}(\mathcal{G}_X^{(2)})=1$ and $\mathrm{ord}_{\xi ^{(1)}}(\mathcal{G}_X^{(1)})=5/2$. Note that $\mathrm{ord}(\varphi )=\mathrm{ord}_{t}(\varphi (z))=1$. On the other hand, $$r_{X,\varphi }=\mathrm{ord}_{t}(\varphi (\mathcal{G}_X^{(3)}))=\mathrm{ord}_{t}(\varphi ^{(2)}(\mathcal{G}_X^{(2)}))=\mathrm{min}\left\{ \mathrm{ord}_t(\varphi ^{(2)}(y)),\mathrm{ord}_{t}(\varphi ^{(1)}(\mathcal{G}_X^{(1)})) \right\} = \mathrm{min}\left\{ 2,5/2 \right\} =2\mbox{.}$$
Hence, for this example $1<\bar{r}_{X,\varphi }=2<5/2$.
\end{Ex}

\vspace{0.2cm}

Let us end our discussion with a couple of illustrative examples for Propositions \ref{prop:sing_small} and \ref{prop:sing_big} respectively. The first one shows an isolated point of $\mathrm{\underline{Max}\, mult}(X)$ for which $\Phi _{X,\xi }$ is upper bounded by $3$:
\vspace{0.2cm}

\begin{Ex}
Let $X=\left\{ x^2y^3-z^3s^4=0 \right\} \hookrightarrow \mathrm{Spec}(k[x,y,z,s])$ and let $\xi =(0,0,0,0)=\mathrm{\underline{Max}\, mult}(X)$. We have
\begin{align*}
&\mathcal{G}_X^{(4)}=\mathrm{Diff}(k[x,y,z,s][(x^2y^3-z^3s^4)W^5])=\\
=k[x,y,z,s][xW,yW,zsW & ,z^3W,s^2W, z^3sW^2,zs^2W^2,zs^4W^3,z^2s^3W^3,z^3s^2W^3,z^3s^3W^4,z^3s^4W^5]\mbox{.}
\end{align*}
Observe now that $xW,yW,z^3W,s^2W\in \mathcal{G}_X^{(4)}$, and $\varphi (x)W,\varphi (y)W,\varphi (z)^3W,\varphi (s)^2W \in \varphi (\mathcal{G}_X^{(4)})$. Then
$$r_{X,\varphi }\leq \mathrm{min}\left\{ \mathrm{ord}_t(\varphi (x)), \mathrm{ord}_t(\varphi (y)), 3\cdot \mathrm{ord}_t(\varphi (z)), 2\cdot \mathrm{ord}_t(\varphi (s)) \right\} \mbox{.}$$
If $\mathrm{ord}(\varphi )=\mathrm{ord}_t(\varphi (x))$ or $\mathrm{ord}(\varphi )=\mathrm{ord}_t(\varphi (y))$, then $\bar{r}_{X,\varphi }=1$. If $\mathrm{ord}(\varphi )=\mathrm{ord}_t(\varphi (z))$, then $\bar{r}_{X,\varphi }\leq 3$, and if $\mathrm{ord}(\varphi )=\mathrm{ord}_t(\varphi (s))$, then $\bar{r}_{X,\varphi }\leq 2$. In any case 
$$\bar{r}_{X,\varphi }\leq 3\mbox{.}$$
\end{Ex}

\vspace{0.2cm}

In our next example we construct, for a non isolated point of $\mathrm{\underline{Max}\, mult}(X)$, a family of arcs $\varphi _N$, $N\in \mathbb{Z}_{>0}$, for which $\bar{r}_{X,\varphi _N}$ equals a polynomial in $N$, namely $q(N)=N+2$, showing that $\Phi _{X,\xi }$ is not upper bounded:

\vspace{0.2cm}

\begin{Ex}
Let now $X=\left\{ x^2y^3-z^4s^5=0 \right\} $, and let $\xi =(0,0,0,0)$ again. Now $\xi \subsetneq \mathrm{\underline{Max}\, mult}(X)$. In this case,
$$\mathcal{G}_X^{(4)}=\mathrm{Diff}(k[x,y,z,s][(x^2y^3-z^4s^5)W^5])=k[x,y,z,s][xW,yW,zsW,s^5W,zs^5W^2,z^2s^5W^3,z^3s^5W^4, z^4s^5W^5]\mbox{.}$$
Consider the following family of arcs through $\xi $ parametrized by $N\in \mathbb{Z}_{>0}$:
\begin{align*}
\varphi _N:k[x,y,z,s]/(x^2y^3-z^4s^5) & \longrightarrow K[[t]]\\
x & \longmapsto t^{2N+2}\mbox{,}\\
y & \longmapsto t^{2N+5}\mbox{,}\\
z & \longmapsto t\mbox{,}\\
s & \longmapsto t^{2N+3}\mbox{.}
\end{align*}
Now 
$$\varphi (\mathcal{G}_X^{(4)})=K[[t]][t^{2N+2}W]$$
and $\mathrm{ord}(\varphi _N)=1$, so
$$\bar{r}_{X,\varphi }=2N+2\mbox{,}$$
which grows with $N$.
\end{Ex}

\end{document}